\documentclass[12pt,a4paper]{article}
\usepackage{mathrsfs}
\usepackage{indentfirst}
\usepackage{latexsym,amsmath,amssymb,amsfonts,amsthm,epsfig,graphicx,cite,psfrag}
\usepackage{makeidx}
\usepackage{color}
\include{PDF}
\newtheorem{theorem}{Theorem}
\newtheorem{lemma}[theorem]{Lemma}
\newtheorem{definition}{Definition}
\newtheorem{corollary}[theorem]{Corollary}

\newtheorem{conjecture}{Conjecture }

\newtheorem{claim}{Claim}

\title{\bf The $3$-rainbow index and connected dominating sets\footnote{Supported by NSFC No.11371205 and
PCSIRT.}}
\author{\small Qingqiong Cai, Xueliang Li, Yan Zhao\\
{\small Center for Combinatorics and LPMC-TJKLC}\\ {\small Nankai
University}\\ {\small Tianjin 300071, China}\\ {\small Email:
cqqnjnu620@163.com, lxl@nankai.edu.cn, zhaoyan2010@mail.nankai.edu.cn
}}
\date{}

\begin{document}
\maketitle

\begin{abstract}
A tree in an edge-colored graph is said to be rainbow if no two
edges on the tree share the same color. An edge-coloring of $G$ is
called 3-rainbow if for any three vertices in $G$, there exists a
rainbow tree connecting them. The 3-rainbow index $rx_3(G)$ of $G$
is defined as the minimum number of colors that are needed in a
3-rainbow coloring of $G$. This concept, introduced by Chartrand et
al., can be viewed as a generalization of the rainbow connection. In
this paper, we study the 3-rainbow index by using connected
three-way dominating sets and 3-dominating sets. We shown that for
every connected graph $G$ on $n$ vertices with minimum degree at
least $\delta$ ($3\leq\delta\leq5$), $rx_{3}(G)\leq
\frac{3n}{\delta+1}+4$, and the bound is tight up to an additive
constant; whereas for every connected graph $G$ on $n$ vertices with
minimum degree at least $\delta$ ($\delta\geq3$), we get that
$rx_{3}(G)\leq n\frac{ln(\delta+1)}{\delta+1}(1+o_{\delta}(1))+5$.
In addition, we obtain some tight upper bounds of the 3-rainbow
index for some special graph classes, including threshold graphs,
chain graphs and interval graphs.
\end{abstract}



\section {\large Introduction}

All graphs in this paper are undirected, finite and simple. We
follow \cite{Bondy} for graph theoretical notation and terminology
not described here. Let $G$ be a nontrivial connected graph with an
\emph{edge-coloring} $c: E(G)\rightarrow\{1, 2,\cdots, t\}, t \in
\mathbb{N}$, where adjacent edges may be colored the same. A path is
said to be a \emph{rainbow path} if no two edges on the path have
the same color. An edge-colored graph $G$ is called \emph{rainbow
connected} if for every pair of distinct vertices of $G$ there
exists a rainbow path connecting them. The \emph{rainbow connection
number} of $G$, denoted by $rc(G)$, is defined as the minimum number
of colors that are needed in order to make $G$ rainbow connected.
The \emph{rainbow $k$-connectivity} of $G$, denoted by $rc_{k}(G)$,
is defined as the minimum number of colors in an edge-coloring of
$G$ such that every two distinct vertices of $G$ are connected by
$k$ internally disjoint rainbow paths. These concepts were
introduced by Chartrand et al. in \cite{ChartrandGP, Chartrand1}.
Recently, there have been published a lot of results on the rainbow
connections. The interested readers can see \cite{LiSun1, LiSun2}
for a survey on this topic.

The $(k,\ell)$-rainbow index was also introduced by Chartrand et al.
in \cite{Chartrand2}, which can be viewed as a generalization of the
rainbow connection and rainbow connectivity. We call a tree $T$ of
an edge-colored graph $G$ a \emph{rainbow tree} if no two edges of
$T$ have the same color. For $S\subseteq V(G)$, a \emph{rainbow
$S$-tree} is a rainbow tree connecting the vertices of $S$. Suppose
that $\{T_{1},T_{2},\cdots, T_{\ell}\}$ is a set of rainbow
$S$-trees. They are called \emph{internally disjoint} if
$E(T_{i})\cap E(T_{j})=\emptyset$ and $V(T_{i})\bigcap V(T_{j})=S$
for every pair of distinct integers $i,j$ with $1\leq i,j\leq \ell$
(Note that these trees are vertex-disjoint in $G\setminus S$). Given
two positive integers $k$, $\ell$ with $k\geq 2$, the
\emph{$(k,\ell)$-rainbow index} $rx_{k,\ell}(G)$ of $G$ is the
minimum number of colors needed in an edge-coloring of $G$ such that
for any set $S$ of $k$ vertices of $G$, there exist $\ell$
internally disjoint rainbow $S$-trees. In particular, for $\ell=1$,
we often write $rx_{k}(G)$ rather than $rx_{k,1}(G)$ and call it the
\emph{$k$-rainbow index}. An edge-coloring of $G$ is called a
\emph{$k$-rainbow coloring} if for any set $S$ of $k$ vertices of
$G$, there exists a rainbow $S$-tree. A simple result for the
$k$-rainbow index \cite{Chartrand2} is that $k-1\leq rx_k(G)\leq
n-1$. It is easy to see that $rx_{2,\ell}(G)=rc_{\ell}(G)$. In the
sequel, we always assume $k\geq 3$.  We refer to
\cite{Cai,Cai2,Cai3,Chen,Li3,Liu} for more details about the
$(k,\ell)$-rainbow index.

Computing the rainbow connection number of a graph is NP-hard
\cite{Chakraborty}, so is computing the $(k,\ell)$-rainbow index.
For this reason, one of the most important goals for studying
rainbow connection number and rainbow index is to obtain good upper
and lower bounds. In the search toward good upper bounds, an idea
that turned out to be successful more than once is considering the
``strengthened" connected dominating set: find a suitable
edge-coloring of the induced graph on such a set, and then extend it
to the whole graph using a constant number of additional colors.

Given a graph $G$, a set $D\subseteq V(G)$ is called a {\it
dominating set} if every vertex of $V\backslash D$ is adjacent to at
least one vertex of $D$. Further, if the subgraph $G[D]$ of $G$
induced by $D$ is connected, we call $D$ a {\it connected dominating
set} of $G$. The {\it domination number} $\gamma(G)$ is the number
of vertices in a minimum dominating set for $G$. Similarly, the {\it
connected domination number} $\gamma_{c}(G)$ is the number of
vertices in a minimum connected dominating set for $G$.

Let $k$ be a positive integer. A dominating set $D$ of $G$ is called
a {\it $k$-way dominating set} if $d(v)\geq k$ for every vertex
$v\in V\setminus D$. In addition, if $G[D]$ is connected, we call
$D$ a {\it connected $k$-way dominating set}. A set $D\subseteq
V(G)$ is called a {\it $k$-dominating set} of $G$ if every vertex of
$V\backslash D$ is adjacent to at least $k$ distinct vertices of
$D$. Furthermore, if $G[D]$ is connected, we call $D$ a {\it
connected $k$-dominating set}. Obviously, a (connected)
$k$-dominating set is also a (connected) $k$-way dominating set, but
the converse is not true.

There have been several results revealing the close relation between
the dominating sets and the rainbow connection number and rainbow
index.

\begin{theorem}\cite{Chandran}\label{thm1}
If $D$ is a connected two-way dominating set of a connected graph $G$,
then $rc(G)\leq rc(G[D])+3$.
\end{theorem}

In \cite{Chandran}, the authors employed Theorem \ref{thm1} to get
some tight upper bounds for the rainbow connection number
of many special graph classes, which were otherwise difficult to obtain.

\begin{theorem}\cite{Liu}\label{thm2}
Let $G$ be a connected graph with minimal degree $\delta(G)\geq 3$.
If $D$ is a connected 2-dominating set of $G$, then $rx_{3}(G)\leq
rx_{3}(G[D])+4$ and the bound is tight.
\end{theorem}

From Theorem \ref{thm2}, the authors determined a tight upper bound
for the 3-rainbow index of the complete bipartite graphs $K_{s,t} \
(3\leq s \leq t)$.

The proofs of the above two theorems are similar. First color the
edges in $G[D]$ using $k$ different colors $\left(k=rc(G[D])\  or \
rx_3(G[D])\right)$. Then select a spanning tree in every connected
component of $H=G-D$. So we construct a spanning forest $F$ of $H$
and choose $X$ and $Y$ as any one of the bipartitions defined by the
forest $F$. Color the edges between $X$ and $D$ and the edges
between $Y$ and $D$ as well as the edges between $X$ and $Y$ with
suitable colors, which gives an edge-coloring we want. Note that in
the process all the edges in $E(H)-E(F)$ are ignored.

In this paper, we will take the edges in $E(H)-E(F)$ into
consideration to get a more subtle coloring strategy. We show that
for a connected graph $G$, $rx_{3}(G)\leq rx_{3}(G[D])+6$, where $D$
is a connected three-way dominating set of $G$. Moreover, this bound
is tight. By using the results on spanning trees with many leaves,
we obtain that $rx_3(G)\leq\frac{3n}{\delta+1}+4$ for every
connected graph $G$ on $n$ vertices with minimum degree at least
$\delta$ ($3\leq\delta\leq5$), and the bound is tight up to an
additive constant; whereas for every connected graph $G$ on $n$
vertices with minimum degree at least $\delta$ ($\delta\geq3$), we
get that $rx_{3}(G)\leq
n\frac{ln(\delta+1)}{\delta+1}(1+o_{\delta}(1))+5$. In addition,
when considering a connected 3-dominating set $D$ of $G$, we prove
that $rx_{3}(G)\leq rx_{3}(G[D])+3$, and the bound is tight. The
farthest we can get with this idea is some tight upper bounds for
some special graph classes, including threshold graphs, chain graphs
and interval graphs.

\section{Preliminaries}
For a graph $G$, we use $V(G)$, $E(G)$, $|G|$, $\delta(G)$, and
$diam(G)$ to denote its vertex set, edge set, order (number of
vertices), minimum degree and the diameter (maximum distance between
every pair of vertices) of $G$, respectively. For $D\subseteq V(G)$,
let $\overline{D}=V(G)\setminus D$, and $G[D]$ be the subgraph of
$G$ induced on $D$. For $v\in V(G)$, let $N(v)$ denote the set of
neighbors of $v$. For two disjoint subsets $X$ and $Y$ of $V(G)$,
$E[X,Y]$ denotes the set of edges of $G$ between $X$ and $Y$.

\begin{definition}
BFS (breadth-first search) is a strategy for searching in a graph.
It begins at a root and inspects all its neighbors. Then for each of
those neighbors in turn, it inspects their neighbors which were
unvisited, and so on until all the vertices in the graph are
visited.
\end{definition}

\begin{definition}
A BFS-tree (breadth-first search tree) is a spanning rooted tree
returned by BFS. Let $T$ be a BFS-tree with $r$ as its root. For a
vertex $v$, the height of $v$ is the distance between $v$ and $r$.
All the vertices of height $k$ form the $kth$ level of $T$. The
ancestors of $v$ are the vertices on the unique $\{v,r\}$-path in
$T$. The parent of $v$ is its neighbor on the unique $\{v,r\}$-path
in $T$. Its other neighbors are called the children of $v$. The
siblings of $v$ are the vertices in the same level as $v$.
\end{definition}

\noindent\textbf{Remark:} $BFS$-trees have a nice property: every
edge of the graph joins vertices on the same or consecutive levels.
It is not possible for an edge to skip a level. Thus the neighbor of
a vertex $v$ has three possibilities: (1) a sibling of $v$; (2) the
parent of $v$ or a right sibling of the parent of $v$; (3) a child
of $v$ or a left sibling of the children of $v$; see Figure 1.

\begin{figure}[ht]
\begin{center}
\includegraphics[width=6cm]{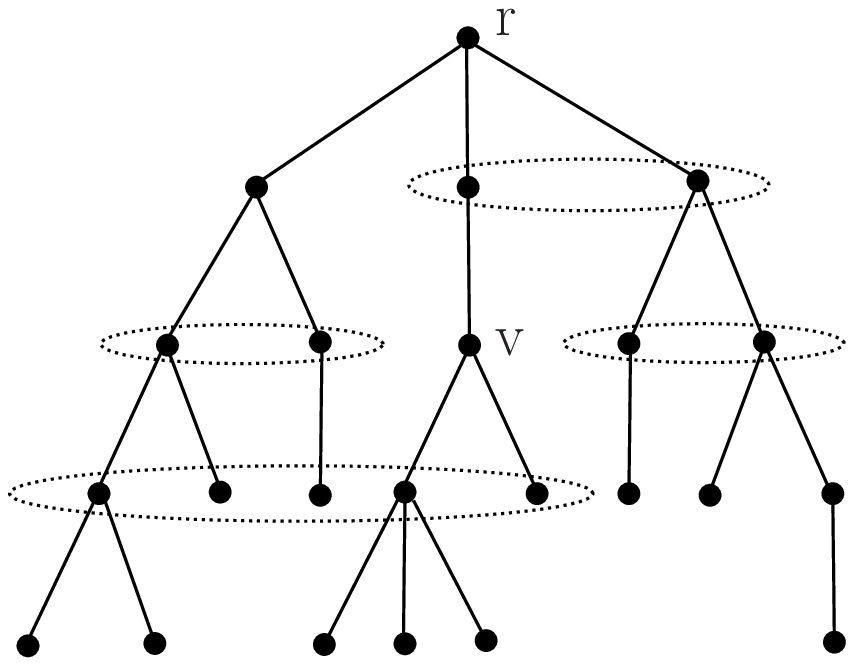}\\
Figure 1: The vertices in the dotted circles are the potential
neighbors of $v$.
\end{center}
\end{figure}

\begin{definition}
The Steiner distance $d(S)$ of a set $S$ of vertices in a graph $G$
is the minimum size of a tree in $G$ containing $S$. The $k$-Steiner
diameter $sdiam_k(G)$ of $G$ is the maximum Steiner distance of $S$
among all the sets $S$ with $k$ vertices in $G$. Obviously,
$sdiam_2(G)=diam(G)$ and $sdiam_k(G)\leq sdiam_{k+1}(G)$.
\end{definition}

\begin{definition}
Let $G$ be a graph, $D\subseteq V(G)$ and $v\in V(G)\setminus D$. we
call a path $P=v_0 v_1\cdots v_k$ is a v-D path if $v_0=v$ and
$V(P)\cap D=\{v_k\}$. Two or more paths are called internally
disjoint if none of them contains an inner vertex of another.
\end{definition}

\begin{definition}
 An edge-colored graph is rainbow if no two edges in the graph share the same color.
\end{definition}

\begin{definition}
Let $D$ be a dominating set of a graph $G$. For $v\in \overline{D}$,
its neighbors in $D$ are called foots of $v$, and the corresponding
edges are called legs of $v$.
\end{definition}

\begin{definition}
A graph $G$ is called a threshold graph, if there exists a weight
function $w: V(G)\rightarrow R$ and a real constant $t$ such that
two vertices $u, v\in V(G)$ are adjacent if and only if
$w(u)+w(v)\geq t$. We call $t$ the threshold for $G$.
\end{definition}

\begin{definition}
A bipartite graph $G(A,B)$ is called a chain graph, if the vertices
of $A$ can be ordered as $A=(a_1,a_2, \ldots ,a_k)$ such that
$N(a_1)\subseteq N(a_2)\subseteq \ldots\subseteq N(a_k)$.
\end{definition}

\begin{definition}
An intersection graph of a family $\mathcal{F}$ of sets is a graph
whose vertices can be mapped to the sets in $\mathcal{F}$ such that
there is an edge between two vertices in the graph if and only if
the corresponding two sets in $\mathcal{F}$ have a non-empty
intersection. An interval graph is an intersection graph of
intervals on the real line.
\end{definition}

\section{Main results}

\begin{theorem}\label{thm3}
If $D$ is a connected three-way dominating set of a connected graph $G$,
 then $rx_{3}(G)\leq rx_{3}(G[D])+6$. Moreover, the bound is tight.
\end{theorem}

The proof of Theorem \ref{thm3} is given in Section 4. Let us first
show how this implies the following results.

\begin{corollary}\label{cor0}
For every connected graph $G$ with $\delta(G)\geq3$,
$rx_{3}(G)\leq \gamma_c(G)+5$.
\end{corollary}

\begin{proof}
In this case, every connected dominating set of $G$ is a connected
three-way dominating set. Now take a minimum connected dominating
set $D$ in $G$. Then $rx_3(G[D])\leq |D|-1=\gamma_c(G)-1$. It
follows from Theorem \ref{thm3} that $rx_3(G)\leq rx_3(G[D])+6\leq
\gamma_c(G)+5$.
\end{proof}

From the following lemma, we can get the next corollary.
\begin{lemma}\label{lem1}

(1)\cite{Kleitman} Every connected graph on $n$ vertices with
minimum degree $\delta\geq3$ has a spanning tree with at least
$\frac{1}{4}n+2$ leaves;

(2)\cite{Griggs} Every connected graph on $n$ vertices with minimum degree
$\delta\geq4$ has a spanning tree with at least $\frac{2}{5}n+\frac{8}{5}$ leaves;

(3)\cite{Griggs} Every connected graph on $n$ vertices with minimum degree
$\delta\geq5$ has a spanning tree with at least $\frac{1}{2}n+2$ leaves;
\end{lemma}

\begin{corollary}\label{cor1}
(1) For every connected graph $G$ on $n$ vertices with $\delta(G)\geq3$,
$rx_{3}(G)\leq \frac{3}{4}n+3$.

(2) For every connected graph $G$ on $n$ vertices with $\delta(G)\geq4$,
$rx_{3}(G)\leq \frac{3}{5}n+\frac{17}{5}$.

(3) For every connected graph $G$ on $n$ vertices with $\delta(G)\geq5$,
$rx_{3}(G)\leq \frac{1}{2}n+3$.

Moreover, these bounds are tight up to an additive constant.
\end{corollary}

\begin{proof}
We only prove (1); (2) and (3) can be derived by the same arguments.

Clearly, we can take a connected dominating set consisting of all
the non-leaves in the spanning tree. Thus by Lemma \ref{lem1}, for
every connected graph $G$ on $n$ vertices with minimum degree
$\delta(G)\geq3$, $\gamma_c(G)\leq
n-(\frac{1}{4}n+2)=\frac{3}{4}n-2$. Then it follows from Corollary
\ref{cor0} that $rx_{3}(G)\leq\frac{3}{4}n+3$.

On the other hand, the factors in these bounds cannot be improved,
since there exist infinitely many graphs $G^*$ such that
$rx_{3}(G^*)\geq \frac{3}{\delta+1}n-\frac{\delta+7}{\delta+1}$. We
construct the graphs as follows (the construction was also mentioned
in \cite{Caro}): first take $m$ copies of $K_{\delta+1}$, denoted by
$X_1, X_2, \ldots, X_m$ and label the vertices of $X_i$ with
$x_{i,1},\ldots, x_{i,\delta+1}$. Then take two copies of
$K_{\delta+2}$, denoted by $X_0$ and $X_{m+1}$ and similarly label
their vertices. Now join $x_{i,2}$ and $x_{i+1,1}$ for
$i=0,1,\ldots, m$ with an edge and delete the edges $x_{i,1}x_{i,2}$
for $i=0,1,\ldots, m+1$. See Figure 2 for $\delta=3$. It is easy to
see that $diam(G^*)=\frac{3}{\delta+1}n-\frac{\delta+7}{\delta+1}$.
The $k$-Steiner diameter of a graph is a trivial lower bound for its
$k$-rainbow index \cite{Chartrand2}, and so $rx_3(G^*)\geq
sdiam_3(G^*)\geq
diam(G^*)=\frac{3}{\delta+1}n-\frac{\delta+7}{\delta+1}$. For
$\delta=3$, $rx_3(G^*)\geq\frac{3}{4}n-\frac{5}{2}$; for $\delta=4$,
$rx_3(G^*)\geq\frac{3}{5}n-\frac{11}{5}$; for $\delta=5$,
$rx_3(G^*)\geq\frac{1}{2}n-2$. Therefore, all these upper bounds are
tight up to an additive constant.
\end{proof}

\begin{figure}[ht]
\begin{center}
\includegraphics[width=10cm]{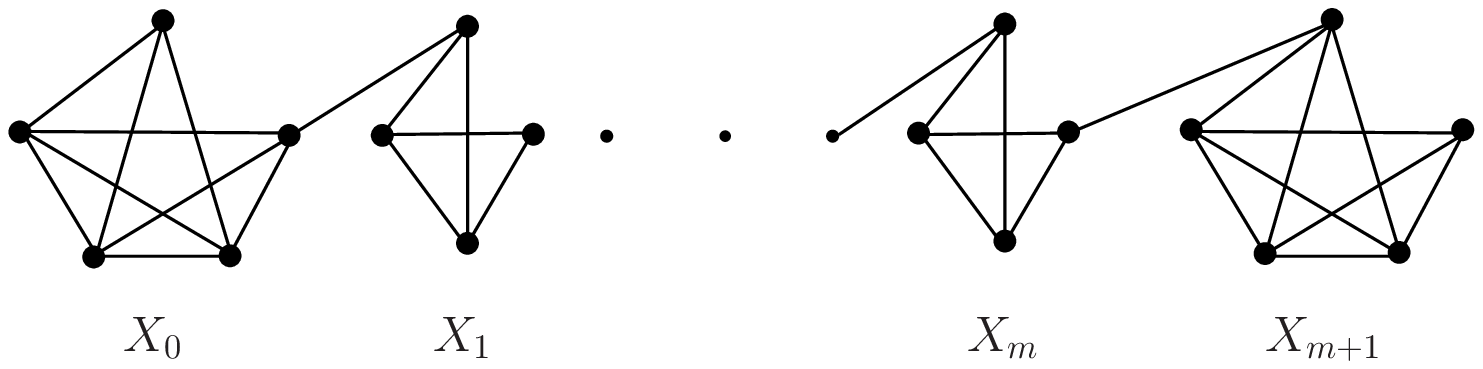}\\
Figure 2: An example for $\delta =3.$
\end{center}
\end{figure}

As to general $\delta$, Caro et. al. \cite{Caro1} proved that for
every connected graph $G$ on $n$ vertices with minimum degree
$\delta$,
$\gamma_c(G)=n\frac{ln(\delta+1)}{\delta+1}(1+o_{\delta}(1))$.
Combining with Corollary \ref{cor0}, we get the following result.

\begin{corollary}\label{cor2}
For every connected graph $G$ on $n$ vertices with minimum degree
$\delta$ ($\delta\geq3$), $rx_{3}(G)\leq
n\frac{ln(\delta+1)}{\delta+1}(1+o_{\delta}(1))+5$.
\end{corollary}

The above bound is not believed to be optimal for $rx_3(G)$ in terms
of $\delta$. We pose the following conjecture, which has already
been proved for $\delta=3,4,5$ in Corollary \ref{cor1}. Note that if
the conjecture is true, it gives an upper bound tight up to an
additive constant by the construction of the graph $G^*$.

\begin{conjecture}
For every connected graph $G$ on $n$ vertices with minimum degree
$\delta$ ($\delta\geq3$), $rx_{3}(G)\leq\frac{3n}{\delta+1}+C$,
where $C$ is a positive constant.
\end{conjecture}

With regard to the graphs possessing vertices of degree 1 or 2, we
obtain the following result.
\begin{corollary}
For every connected graph $G$, $rx_3(G)\leq \gamma_c(G)+n_1+n_2+5$,
where $n_1$ and $n_2$ denote the number of vertices of degrees 1 and
2 in $G$, respectively.
\end{corollary}

\begin{proof}
Obviously, adding all the vertices of degrees 1 and 2 into a minimum
connected dominating set forms a connected three-way dominating set
in $G$ of size no more than $\gamma_c(G)+n_1+n_2$. Consequently, by
Theorem \ref{thm3}, $rx_3(G)\leq \gamma_c(G)+n_1+n_2+5$.
\end{proof}

We proceed with another upper bound for the 3-rainbow index of
graphs concerning the connected 3-dominating set.

\begin{theorem}\label{thm4}
If $D$ is a connected 3-dominating set of a connected graph $G$ with $\delta(G)\geq3$,
 then $rx_{3}(G)\leq rx_{3}(G[D])+3$. Moreover, the bound is tight.
\end{theorem}

\begin{proof}
Since $D$ is a connected 3-dominating set, every vertex in
$\overline{D}$ has at least three legs. Color one of them with 1,
one of them with 2, and all the others with 3. Let $k=rx_3(G[D])$.
Then we can color the edges in $G[D]$ with $k$ different colors from
$\{4,5,\ldots, k+3\}$ such that for every triple of vertices in $D$,
there exists a rainbow tree in $G[D]$ connecting them. If there
remain uncolored edges in $G$, we color them with 1.

Next we will show that this edge-coloring is a 3-rainbow coloring of
$G$. For any triple $\{u,v,w\}$ of vertices in $G$, if $(u,v,w)\in
D\times D\times D$, then there is already a rainbow tree connecting
them in $G[D]$. If one of them is in $\overline{D}$, say $(u,v,w)
\in \overline{D}\times D\times D$, join any leg of $u$ (colored by
1, 2, or 3) with the rainbow tree connecting $v,w$ and the
corresponding foot of $u$ in $G[D]$. If two of them are in
$\overline{D}$, say $(u,v,w) \in \overline{D}\times
\overline{D}\times D$, join one leg of $u$ colored by $1$, one leg
of $v$ colored by $2$ with the rainbow tree connecting $w$ and the
corresponding foots of $u,v$ in $G[D]$. If $(u,v,w)\in
\overline{D}\times \overline{D}\times \overline{D}$, join one leg of
$u$ colored by $1$, one leg of $v$ colored by $2$, one leg of $w$
colored by $3$ with the rainbow tree connecting the corresponding
foots of $u,v,w$ in $G[D]$.

The tightness of the bound can be seen from the next Corollary.
\end{proof}

As immediate consequences of Theorem \ref{thm3} and Theorem
\ref{thm4}, we get the following:

\begin{corollary}
Let $G$ be a connected graph with $\delta(G)\geq3$.

(1) if $G$ is a threshold graph, then $rx_{3}(G)\leq 5$;

(2) if $G$ is a chain graph, then $rx_{3}(G)\leq 6$;

(3) if $G$ is an interval graph, then $rx_{3}(G)\leq diam(G)+4$.
Thus $diam(G) \leq rx_{3}(G)\leq diam(G)+4$;

Moreover, all these upper bounds are tight.
\end{corollary}

\begin{proof}
(1) Suppose that $V(G)=\{v_1,v_2,\ldots, v_n\}$ where $w(v_1)\geq
w(v_2)\geq\ldots\geq w(v_n)$. Since the minimum degree of $G$ is at
least three, $v_i$ ($1\leq i\leq 3$) is adjacent to all the other
vertices in $G$. Thus $D=\{v_1,v_2,v_3\}$ consists of a connected
3-dominating set of $G$. Note that $D$ induces a $K_3$, so
$rx_3(G[D])=2$. It follows from Theorem \ref{thm4} that $rx_3(G)\leq
rx_3(G[D])+3=5$.

(2) Suppose that $G=G(A,B)$ and the vertices of $A$ can be ordered
as $A=(a_1,a_2, \ldots ,a_k)$ such that $N(a_1)\subseteq
N(a_2)\subseteq \ldots\subseteq N(a_k)$. Since the minimum degree of
$G$ is at least three, $a_i$ ($k-2\leq i\leq k$) is adjacent to all
the vertices in $B$, and $N(a_1)$ has at least three vertices, say
$\{b_1,b_2,b_3\}$. Clearly $b_i$ ($1\leq i\leq 3$)
is adjacent to all the vertices in $A$. Thus
$D=\{a_{k-2},a_{k-1},a_{k}, b_1,b_2,b_3\}$ consists of a connected
3-dominating set of $G$. Note that $D$ induces a $K_{3,3}$, so
$rx_3(G[D])=3$ (see \cite{Chen}). It follows from Theorem \ref{thm4} that $rx_3(G)\leq
rx_3(G[D])+3=6$.

(3) If $G$ is isomorphic to a complete graph, then $rx_3(G)=2$ or
$3$ (see \cite{Chartrand2}), the assertion holds trivially. Otherwise, it was showed in
\cite{Chandran} that every interval graph $G$ which is not
isomorphic to a complete graph has a dominating path $P$ of length
at most $diam(G)-2$. Since $\delta(G)\geq3$, $P$ consists of a
connected three-way dominating set of $G$. It follows from Theorem
\ref{thm3} that $rx_3(G)\leq rx_3(P)+6\leq diam(G)+4$. On the other
hand, $rx_3(G)\geq sdiam_3(G) \geq diam_(G)$. We conclude that for a
connected interval graph $G$ with $\delta\geq3$, $diam(G)\leq
rx_3(G)\leq diam(G)+4$

Here we give examples to show the tightness of these upper bounds.

\noindent\textbf{Example 1:}  A threshold graph $G$ with
$\delta(G)\geq3$ and $rx_3(G)=5$.

\begin{figure}[ht]
\begin{center}
\includegraphics[width=7cm]{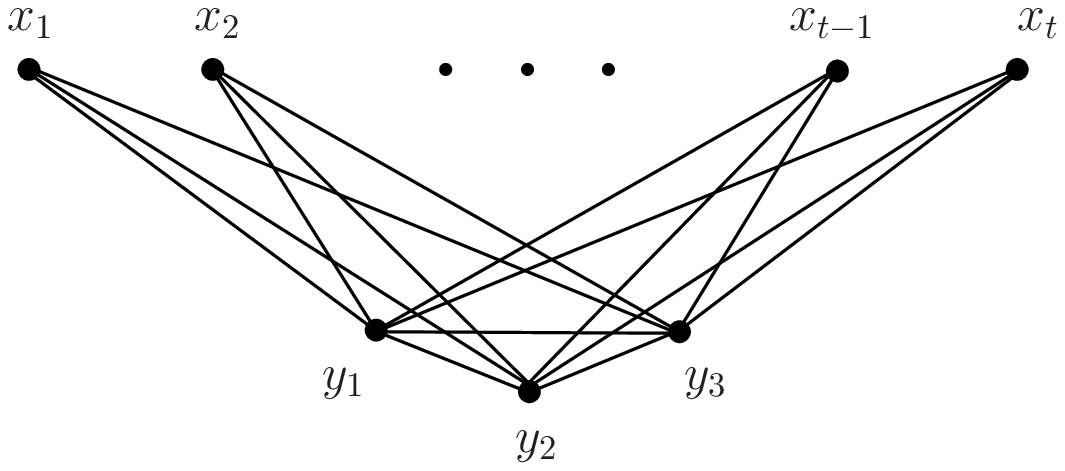}\\
Figure 3:  A threshold graph $G$ with $\delta(G)\geq3$ and
$rx_3(G)=5.$
\end{center}
\end{figure}

Consider the graph in $Figure \ 3$, where $t\geq2\times4^3+1$. It is
easy to see that it is a threshold graph ($y_1, y_2, y_3$ can be
given a weight 1, others a weight 0 and the threshold 1). By
contradiction, we assume that $G$ can be colored with 4 colors. Let
$S_i$ denote the star with $x_i$ as its center and $E(S_i) =
\{x_iy_1, x_iy_2, x_iy_3\}$. Every $S_i$ can be colored in $4^3$
different ways. Since $t\geq 2\times4^3+1$, there exist three
completely identical edge-colored stars, say $S_1$, $S_2$ and $S_3$.
If two of the three edges in $S_i$ $(1\leq i\leq 3)$ receive the
same color, then there are no rainbow trees connecting
$x_1,x_2,x_3$, a contradiction. If the three edges in $S_i$ $(1\leq
i\leq 3)$ receive distinct colors, then the rainbow tree connecting
$x_1,x_2,x_3$ must contain the vertices $y_1,y_2,y_3$. Thus the tree
has at least five edges, but only four different colors, a
contradiction.

\noindent\textbf{Example 2:}  A chain graph $G$ with $\delta(G)\geq3$ and $rx_3(G)=6$.

\begin{figure}[ht]
\begin{center}
\includegraphics[width=10cm]{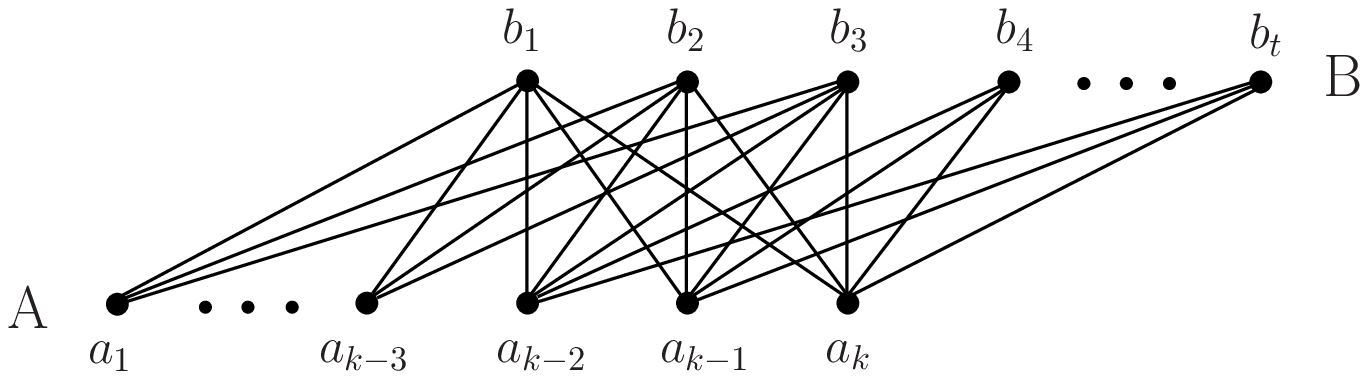}\\
Figure  4: A chain graph $G$ with $\delta(G)\geq3$ and $rx_3(G)=6.$
\end{center}
\end{figure}

Consider the bipartite graph in $Figure \ 4$, where $N(a_1)=N(a_2
=\cdots=N(a_{k-3})=\{b_1, b_2, b_3\}$,
$N(a_{k-2})=N(a_{k-1})=N(a_{k})=\{b_1, b_2, \cdots, b_t\}$, and
$t\geq 2\times5^3+4$. By contradiction, we assume that $G$ can be
colored with 5 colors. Let $S_i$ ($4\leq i\leq t$) denote the star
with $b_i$ as its center and $E(S_i) = \{b_ia_{k-2}, b_ia_{k-1},
b_ia_{k}\}$. Every $S_i$ can be colored in $5^3$ different ways.
Since $t-3\geq 2\times5^3+1$, among the $t-3$ $S_i$$'s$ there exist
three completely identical edge-colored stars, say $S_4$, $S_5$ and
$S_6$. If two of the three edges in $S_i$ $(4\leq i\leq 6)$ receive
the same color, then there are no rainbow trees connecting
$b_4,b_5,b_6$, a contradiction. If the three edges in $S_i$ $(4\leq
i\leq 6)$ receive distinct colors, then the rainbow tree connecting
$b_4,b_5,b_6$ must contain $a_{k-2},a_{k-1},a_{k}$ and at least one
vertex in $B\setminus \{b_4,b_5,b_6\}$ to connect
$a_{k-2},a_{k-1},a_{k}$. Thus the tree has at least six edges, but
only five different colors, a contradiction.

\noindent\textbf{Example 3:}  An interval graph $G$ with
$\delta(G)\geq3$ and $rx_3(G)=diam(G)+4$.

\begin{figure}[ht]
\begin{center}
\includegraphics[width=6cm]{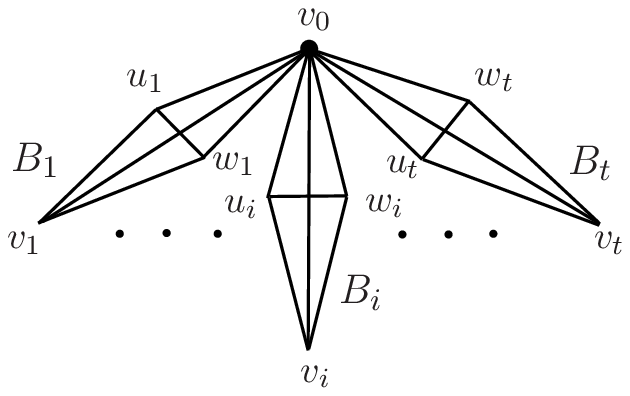}\\
Figure 5 An interval graph $G$ with $\delta(G)\geq3$ and
$rx_3(G)=diam(G)+4.$
\end{center}
\end{figure}

Consider the graph in $Figure \ 5$ (it is known as a French
Windmill), where $t\geq 2\times 5^6+1$. It is easy to see that it is
an interval graph with diameter 2. It follows from (3) that
$rx_3(G)\leq diam(G)+4=6$. We will show that $rx_3(G)=6$. By
contradiction, we assume that $G$ can be colored with 5 colors. Let
$B_i$ denote the $K_4$ induced by $v_0,u_i,v_i,w_i$. Obviously, each
$B_i$ can be colored in at most $5^6$ different ways. Since $t\geq
2\times 5^6+1$, there exist three completely identical edge-colored
subgraphs, say $B_1, B_2, B_3$. If two of the three edges incident
with $v_0$ in $B_i$ ($1 \leq i\leq 3$) receive the same color, say
$c(v_0u_i)=c(v_0v_i)=1$, then there are no rainbow trees connecting
$u_1,u_2,u_3$, a contradiction. If the three edges incident with
$v_0$ in $B_i$ ($1 \leq i\leq 3$) receive distinct colors, say
$c(v_0u_i)=1, c(v_0v_i)=2, c(v_0w_i)=3$, then  $c(u_iv_i)\neq
c(u_iw_i)$ and $c(u_iv_i), c(u_iw_i)\in \{4,5\}$ because there
exists a rainbow tree connecting $\{u_1,u_2,u_3\}$. Without loss of
generality, suppose $c(u_iv_i)=4$ and $c(u_iw_i)=5$ . Since there
exists a rainbow tree connecting $\{v_1,v_2,v_3\}$, then
$c(v_iw_i)=5$. But then there exist no rainbow trees connecting
$\{w_1,w_2,w_3\}$ in $G$, a contradiction.
\end{proof}

\section{Proof of Theorem 3}
Let $D$ be a connected three-way dominating set of a connected graph
$G$. We want to show that $rx_3(G)\leq rx_3(G[D])+6$.

To start with, we introduce some definitions and notation that are
used in the sequel. A set of rainbow paths
$\{P_1,P_2,\ldots,P_\ell,\}$ is called \emph{super-rainbow} if their
union $\bigcup_{i=1}^{\ell}P_{i}$ is also rainbow. For a vertex $v$
in $\overline{D}$, we call it \emph{safe} if there are three
internally disjoint super-rainbow $v-D$ paths. Otherwise, we call
$v$ \emph{dangerous}. Let $c(e)$ be the color of an edge $e$, $c(H)$
the set of colors appearing on the edges in a graph $H$. For a
vertex $v$ in a $BFS$-tree, we denote the height of $v$ by $h(v)$,
the parent of $v$ by $p(v)$, the child of $v$ by $ch(v)$, the
ancestor of $v$ in the first level by $\pi(v)$.

Let us overview our idea: firstly, we aim to color the edges in
$E[D,\overline{D}]$ and $E(G[\overline{D}])$ with six different
colors. Our coloring strategy has two steps: in the first step, we
give a periodical coloring on some edges in $E[D,\overline{D}]$ and
$E(G[\overline{D}])$. And then most vertices in $\overline{D}$
become safe; in the second step, we color the carefully chosen
uncolored edges and recolor some colored edges intelligently to
ensure that all the vertices in $\overline{D}$ are safe. Then we
extend the coloring to the whole graph by coloring the edges in
$G[D]$ with $rx_3(G[D])$ fresh colors. Finally, we will show that
this edge-coloring of $G$ is a 3-rainbow coloring, which implies
$rx_3(G)\leq rx_3(G[D])+6$.

\subsection{Color the edges in $E[D,\overline{D}]$ and $E(G[\overline{D}])$}
\noindent\textbf{4.1.1~  First step: a periodical coloring}

Assume that $C_1,C_2,\ldots,C_q$ are the connected components of the
subgraph $G-D$.

If $C_i$ $(1\leq i\leq q)$ consists of an isolated vertex $v$, then
$v$ has at least three legs. We color one of them with 1, one of
them with 2, and all the others with 3. Note that now $v$ is safe.

If $C_i$ $(1\leq i\leq q)$ consists of an isolated edge $uv$, then
$u$ has at least two legs. We color one of them with 1, and all the
others with 2. Similarly, $v$ has at least two legs. We color one of
them with 2, and all the others with 3. And color $uv$ with 4. Note
that now both $u$ and $v$ are safe.

If $C_i$ $(1\leq i\leq q)$ consists of at least three vertices, then
there exists a vertex $v_0$ in $C_i$ possessing at least two
neighbors in $C_i$. Starting from $v_0$, we construct a $BFS$-tree
$T$ of $C_i$. Suppose the neighbors of $v_0$ in $C_i$ are
$\{v_1,v_2,\ldots,v_k\}$ ($k\geq2$), which forms the first level of
$T$. For each vertex $v$ in $C_i$, let $e_v$ be one leg of $v$ (if
there are many legs, we pick one arbitrarily), $t(v)$ the
corresponding foot of $v$, $f_v$ the unique edge joining $v$ and its
parent in $T$.

Now we color the edges $e_v$ and $f_v$ as follows: $c(e_{v_0})=2$;
$c(f_{v_i})=4$ and $c(e_{v_i})=1$ for $1\leq i\leq k-1$;
$c(f_{v_k})=5$ and $c(e_{v_k})=3$; for each vertex $v$ in
$V(C_i)\setminus \{v_0,v_1,\ldots,v_k\}$, if $\pi(v)=v_k$, then set
$c(f_v)=4$ and $c(e_v)=2$  when $h(v)\equiv 0 \ (mod\ 3)$,
$c(f_v)=5$ and $c(e_v)=3$ when $h(v)\equiv 1 \ (mod\ 3)$, $c(f_v)=6$
and $c(e_v)=1$ when $h(v)\equiv 2 \ (mod\ 3)$; otherwise, if
$\pi(v)=v_i(1\leq i\leq k-1)$, then set $c(f_v)=6$ and $c(e_v)=2$
when $h(v)\equiv 0 \ (mod\ 3)$, $c(f_v)=4$ and $c(e_v)=1$ when
$h(v)\equiv 1 \ (mod\ 3)$, $c(f_v)=5$ and $c(e_v)=3$ when
$h(v)\equiv 2 \ (mod\ 3)$. In fact, this gives a periodical coloring
depicted as Figure 6.

We call the subtree of $T$ rooted at $v_i \ (1\leq i\leq k-1)$ of
type $I$ and the subtree of $T$ rooted at $v_k$ of type
$\uppercase\expandafter{\romannumeral2}$. There may be many subtrees
of type $I$, but only one subtree of type
$\uppercase\expandafter{\romannumeral2}$. The subtrees of the same
type are colored in the same way. More precisely, if two vertices
$u,v$ lie in the same level and belong to subtrees of the same type,
then $c(e_u)=c(e_v)$ and $c(f_u)=c(f_v)$ after first step.

\begin{figure}[ht]
\begin{center}
\includegraphics[width=5.5cm]{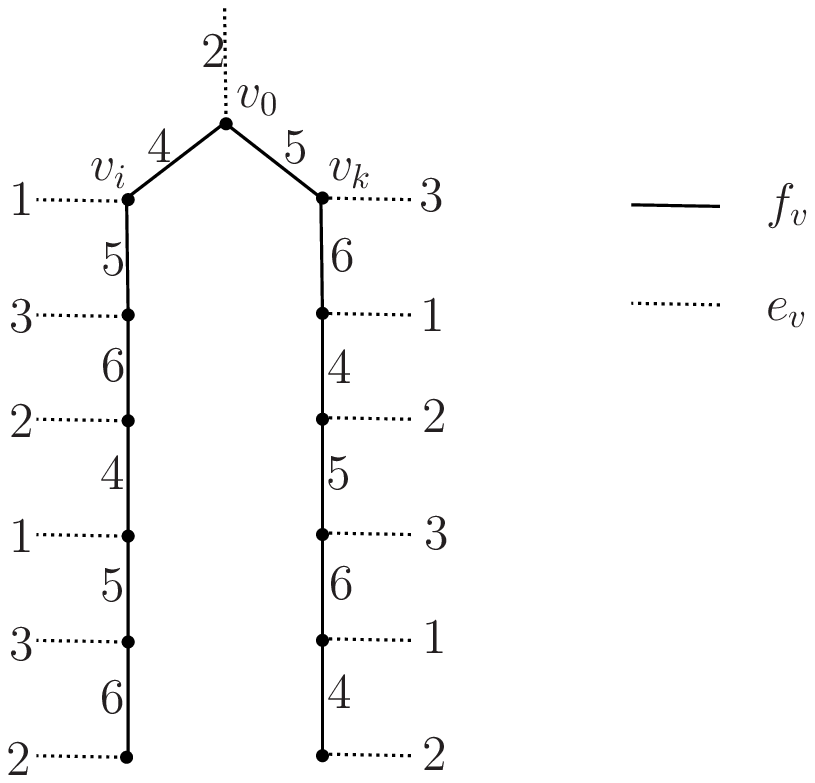}\\
Figure 6: The left branch represents the coloring of subtrees of
type $I$ and the right branch represents the coloring of the subtree
of type $\uppercase\expandafter{\romannumeral2}$.
\end{center}
\end{figure}

Now each non-leaf vertex in $T$ has three internally disjoint
super-rainbow paths connecting it to $D$: for the root $v_0$,
$P_1^{v_0}=v_0,t(v_0)$; $P_2^{v_0}=v_0,v_1,t(v_1)$;
$P_3^{v_0}=v_0,v_k,t(v_k)$. for other non-leaf vertex $v$ in $T$,
$P_1^v=v,t(v)$; $P_2^v=v,p(v),t(p(v))$; $P_3^v=v,ch(v),t(ch(v))$.
(Note that $v$ may have many children $u_1, u_2,\ldots, u_{\ell}$,
but all the $e_{u_i}$$'s$, $f_{u_i}$$'s$ receive the same color. So
they only contribute one path to the three internally disjoint
super-rainbow $v-D$ paths.) In other words, after first step all the
non-leaf vertices in $T$ are safe .

As to each leaf $v'$ in $T$, since $v'$ has no children, it has
exactly two internally disjoint super-rainbow $v'-D$ paths:
 $P_1^{v'}=v',t(v')$;
 $P_2^{v'}=v',p(v'),t(p(v'))$.
 In other words, after first step all the leaves in $T$ are dangerous.

\noindent\textbf{Example 4:} The root $v_0$ is safe:
$c(P_1^{v_0})=\{2\}$, $c(P_2^{v_0})=\{1,4\}$,
$c(P_3^{v_0})=\{3,5\}$.

If $v_i$ $(1\leq i\leq k-1)$ is not a leaf of $T$, then $v_i$ is
safe: $c(P_1^{v_i})=\{1\}$, $c(P_2^{v_i})=\{2,4\}$,
$c(P_3^{v_i})=\{3,5\}$.

If $v_k$ is not a leaf of $T$, then $v_k$ is safe:
$c(P_1^{v_k})=\{3\}$, $c(P_2^{v_k})=\{2,5\}$,
$c(P_3^{v_k})=\{1,6\}$.

\noindent\textbf{Example 5:} If $v$ is a leaf of $T$ in the second
level with parent $v_k$, then $v$ is dangerous: $c(P_1^{v})=\{1\}$,
$c(P_2^{v})=\{3,6\}$.

All the possible color sets of the three internally disjoint
super-rainbow paths connecting a non-leaf vertex to $D$ are: (the
first part in every brace is the color of $P_1$, the second is the
color of $P_2$, the third is the color of $P_3$)

$\{1,24,35\}$,~~~$\{2,36,14\}$,~~~$\{3,15,26\}$,
~~~$\{1,36,24\}$,~~~$\{2,14,35\}$,~~~$\{3,25,16\}$.

Bearing in mind that $D$ is a connected three-way dominating set,
each leaf in $T$ is incident with at least one uncolored edge. In
second step, we will color such edges and recolor some colored edges
suitably to ensure that all the vertices in $C_i$ are safe.

\noindent\textbf{4.1.2~  Second step: more edges with a more
intelligent coloring}

Let $v$ be a leaf in $T$ and $g_v=vv'$ be one uncolored edge incident
with $v$.

If $g_v$ connects $v$ to $D$, then we give $g_v$ a smallest color
from $\{1,2,3,4,5,6\}\setminus (c(P_1^v)\cup c(P_2^v))$. For
instance, $c(g_v)=2$ for the vertex $v$ in Example 5. Obviously, now
$v$ has three internally disjoint super-rainbow $v-D$ paths
$P_1^v,P_2^v,P_3^v$, where
$P_1^v=v,t(v)$; $P_2^v=v,v'$; $P_3^v=v,p(v),t(p(v))$.
In other words, $v$ is safe after second step. All the possible color sets of
the three internally disjoint super-rainbow paths connecting $v$ to
$D$ are: (the first part in every brace is the color of $P_1$, the
second is the color of $P_2$, the third is the color of $P_3$)

$\{1,2,36\}$, ~~$\{1,3,24\}$, ~~$\{1,3,25\}$, ~~$\{2,3,15\}$,
~~$\{2,3,14\}$.

Now it remains to deal with the leaves in $T$ whose incident uncolored edges
all lie in $C_i$. Let $A$ denote the set of such vertices. First, we
flag all the vertices in $V(C_i)\setminus A$, which are already
safe. Note that we only flag the safe vertices. Once one vertex gets
flagged, it is always flagged. Next we arrange the vertices in $A$
in a linear order by the following three rules:

(R1) for $u,v\in A$, let $\pi(u)=v_i$ and $\pi(v)=v_j$, if $i>j$,
then $u$ is before $v$ in the ordering;

(R2) if $\pi(u)=\pi(v)$ and $h(v)>h(u)$, then $u$ is before $v$ in
the ordering.

(R3) if $\pi(u)=\pi(v)$, $h(u)=h(v)$ and $u$ is reached earlier than
$v$ in the $BFS$-algorithm, then $u$ is before $v$ in the ordering.

Assume the vertices in $A$ are ordered as $A=(w_1,w_2,\ldots,w_s)$.
We will deal with them one by one. Suppose that now we go to the
vertex $w_i$ ($w_1,w_2,\ldots,w_{i-1}$ have been processed). If
$w_i$ is flagged, we go to the next vertex $w_{i+1}$; otherwise,
we distinguish the following four cases:

\emph{\textbf{Case 1}}: $\pi(w_i)=v_k$ and there exists at least one
uncolored edge connecting $w_i$ to some subtree of type $I$. Then we
choose one such edge $w_iv$ such that the height of $v$ is as small
as possible. Since $T$ is a $BFS$-tree and the subtree of $w_i$ is
to the right of the subtree of $v$, then $h(v)=h(w_i)$ or
$h(w_i)+1$.

\noindent\textbf{Fact 1.} $e_v$ is not recolored.

If $v\notin A$, then $e_v$ never gets recolored; if $v\in A$, since
$\pi(w_i)=v_k$ and $\pi(v)=v_j$ $(1\leq j\leq k-1)$, we have not
dealt with $v$ yet according to R1, thus $e_v$ is not recolored.

We distinguish three subcases based on the height of $w_i$.

$\ast$ \emph{Subcase 1.1}: $h(w_i)\equiv 0(mod\ 3)$

If $h(v)=h(w_i)$, then color $w_iv$ with $5$. We have
$c(P_1^{w_i})\cup c(P_2^{w_i})\cup c(P_3^{w_i})=\{2\}\cup
\{1,4\}\cup \{3,5,6\}$ and $c(P_1^{v})\cup c(P_2^{v})\cup
c(P_3^{v})=\{2\}\cup \{3,6\}\cup \{1,4,5\}$.

If $h(v)=h(w_i)+1$, then color $w_iv$ with $5$. We have
$c(P_1^{w_i})\cup c(P_2^{w_i})\cup c(P_3^{w_i})=\{2\}\cup
\{3,4,6\}\cup\{1,5\}$ and $c(P_1^{v})\cup c(P_2^{v})\cup
c(P_3^{v})=\{1\}\cup \{3,4,6\}\cup\{2,5\}$.

Now both $w_i$ and $v$ become safe. We flag $w_i$ and $v$ (if $v$ is
not flagged).

$\ast$ \emph{Subcase 1.2}: $h(w_i)\equiv 1(mod\ 3)$

If $h(v)=h(w_i)$, then color $w_iv$ with $6$. We have
$c(P_1^{w_i})\cup c(P_2^{w_i})\cup c(P_3^{w_i})=\{3\}\cup
\{2,5\}\cup \{1,6\}$ and $c(P_1^{v})\cup c(P_2^{v})\cup
c(P_3^{v})=\{1\}\cup \{2,4\}\cup \{3,6\}$.

If $h(v)=h(w_i)+1$, then color $w_iv$ with $4$ and recolor $e_{w_i}$
with $6$. In this way, we ensure that the parent of $w_i$ is still
safe. Now $c(P_1^{p(w_i)})\cup c(P_2^{p(w_i)})\cup
c(P_3^{p(w_i)})=\{2\}\cup \{1,4\}\cup \{5,6\}$. Moreover, we have
$c(P_1^{w_i})\cup c(P_2^{w_i})\cup c(P_3^{w_i})=\{6\}\cup
\{2,5\}\cup \{3,4\}$ and $c(P_1^{v})\cup c(P_2^{v})\cup
c(P_3^{v})=\{3\}\cup \{1,5\}\cup \{4,6\}$.

Now both $w_i$ and $v$ become safe. We flag $w_i$ and $v$ (if $v$ is
not flagged).

$\ast$ \emph{Subcase 1.3}: $h(w_i)\equiv 2(mod\ 3)$

If $h(v)=h(w_i)$, then color $w_iv$ with $2$ and recolor $e_{w_i}$
with $4$. In this way, we ensure that the parent of $w_i$ is still
safe. Now $c(P_1^{p(w_i)})\cup c(P_2^{p(w_i)})\cup
c(P_3^{p(w_i)})=\{3\}\cup \{2,5\}\cup \{4,6\}$. Moreover, we have
$c(P_1^{w_i})\cup c(P_2^{w_i})\cup c(P_3^{w_i})=\{4\}\cup
\{3,6\}\cup \{1,2,5\}$ and $c(P_1^{v})\cup c(P_2^{v})\cup
c(P_3^{v})=\{3\}\cup \{1,5\}\cup \{2,4\}$.

If $h(v)=h(w_i)+1$, then color $w_iv$ with $5$. We have
$c(P_1^{w_i})\cup c(P_2^{w_i})\cup c(P_3^{w_i})=\{1\}\cup
\{3,6\}\cup \{2,5\}$ and $c(P_1^{v})\cup c(P_2^{v})\cup
c(P_3^{v})=\{2\}\cup \{3,6\}\cup \{1,5\}$.

Now both $w_i$ and $v$ become safe. We flag $w_i$ and $v$ (if $v$ is
not flagged).

\noindent\textbf{Remarks:} 1. When dealing with $w_i$, we just do
two operations: (i) coloring $w_iv$; (ii) recoloring $e_{w_i}$ if
necessary. Note that $e_{w_i}$ is the only edge which may be
recolored in this process. Furthermore, we recolor it in such a way
that the parent of $w_i$ is still safe. In fact, for that sake, we
have no choice but to recolor $e_{w_i}$ ($w_i\notin
\{v_1,v_2,\ldots,v_{k-1}\}$) with the unique color which is from
$\{1,2,3,4,5,6\}$ but does not appear on the three super-rainbow
paths of $p(w_i)$ after the first step. For example, in $Subcase\
1.2$, the color set of the three super-rainbow paths of $p(w_i)$
after first step is $\{2,1,4,5,3\}$, so we recolor $e_{w_i}$ with
$6$. The exception that $w_i\in \{v_1,v_2,\ldots,v_{k-1}\}$ will be
discussed in $Subcase\ 3.2$.

2. One may wonder what is the effect of these operations. First of
all, after the process, $w_i$ becomes safe and gets flagged, and so
does $v$ if $v$ is not flagged. In addition, the process guarantees
that all the safe vertices remain safe. As mentioned above, $p(w_i)$
is still safe after this process. For every other safe vertex in
$V(C_i)\setminus A$, obviously its three internally disjoint
super-rainbow paths do not contain $e_{w_i}$, so it is still safe
after this process. For each safe vertex $v$ in $A$, if its three
internally disjoint super-rainbow paths  contain $e_{w_i}$, $w_i$ is
already safe and gets flagged before dealing with it. Then we go to
$w_{i+1}$ directly without doing this process. So we claim that the
three internally disjoint super-rainbow paths of $v$ do not contain
$e_{w_i}$, and thus it is still safe after this process.

3. The three internally disjoint super-rainbow paths of $w_i$ is one
of the following three cases; see Figure 7.

~~~(i) $P_1^{w_i}=w_i,t(w_i)$, $P_2^{w_i}=w_i,p(w_i),t(p(w_i))$,
$P_3^{w_i}=w_i,v,t(v)$;

~~~(ii) $P_1^{w_i}=w_i,t(w_i)$, $P_2^{w_i}=w_i,p(w_i),t(p(w_i))$,
$P_3^{w_i}=w_i,v,p(v),t(p(v))$;

~~~(iii) $P_1^{w_i}=w_i,t(w_i)$, $P_2^{w_i}=w_i,p(w_i), p(p(w_i)),
t(p(p(w_i)))$, $P_3^{w_i}=w_i,v,t(v)$.

\begin{figure}[ht]
\begin{center}
\includegraphics[width=7cm]{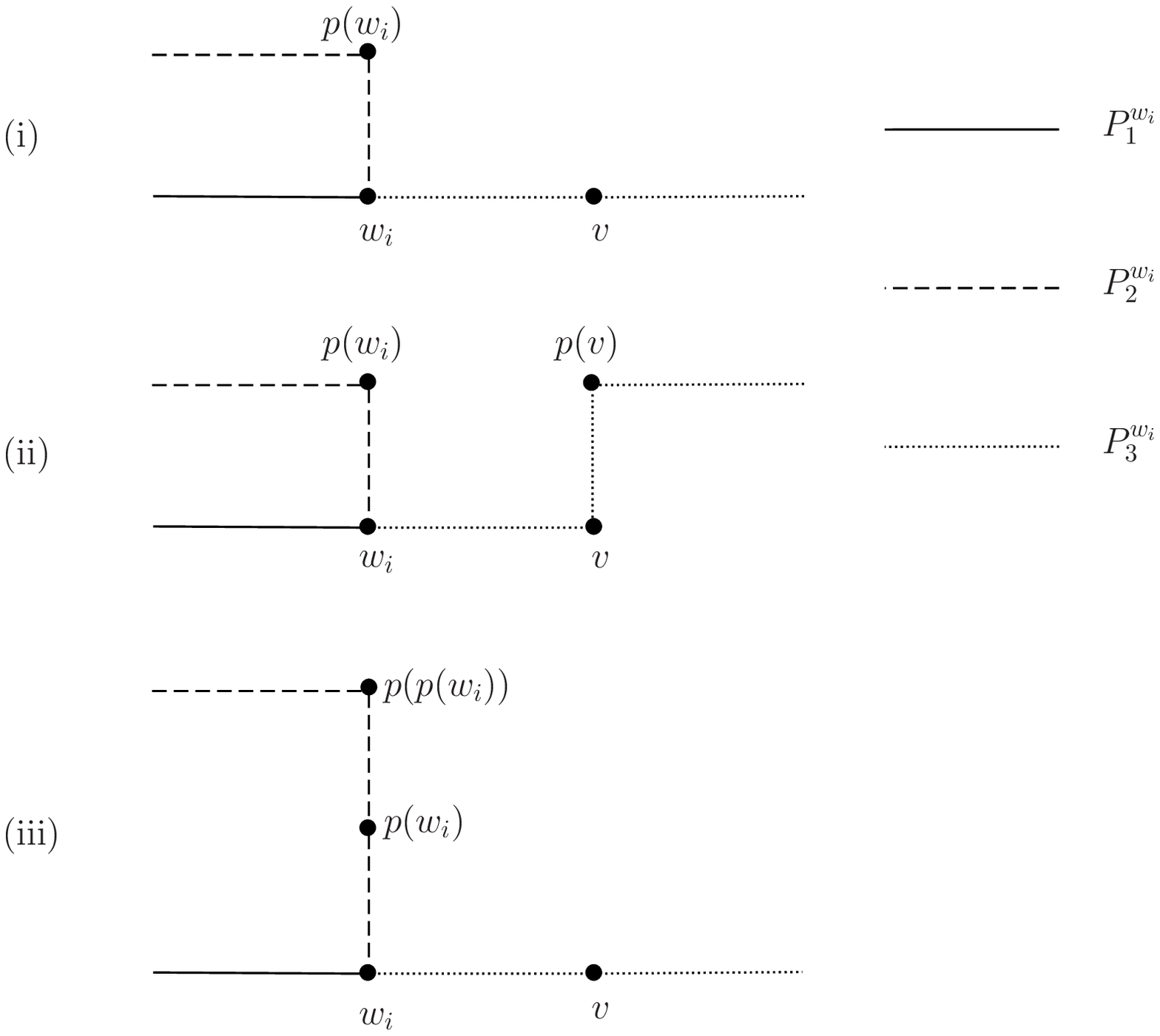}\\
Figure 7: (3) of the Remarks.
\end{center}
\end{figure}

\emph{\textbf{Case 2}}: $\pi(w_i)=v_k$ and all the uncolored edges
connect $w_i$ to the subtree of type
$\uppercase\expandafter{\romannumeral2}$. Then we choose one such
edge $w_iv$ such that the height of $v$ is as small as possible.
Since $T$ is a $BFS-$tree, we get $h(v)=h(w_i)-1$, $h(w_i)$, or
$h(w_i)+1$. The following two facts are easy to see:

\noindent\textbf{Fact 2.} If $h(v)=h(w_i)-1$, then $v$ is already
flagged.

If $v\notin A$, then $v$ gets flagged at the very beginning; if $v\in A$,
since $\pi(v)=\pi(w_i)=v_k$ and $h(v)<h(w_i)$, we have already dealt
with $v$ according to R2, thus $v$ is flagged (note that $e_v$ may
be recolored).

\noindent\textbf{Fact 3.} If $h(v)=h(w_i)+1$, then $e_v$ is not
recolored.

If $v\notin A$, then $e_v$ never gets recolored; if $v\in A$, since
$\pi(v)=\pi(w_i)=v_k$ and $h(v)>h(w_i)$, we have not dealt with $v$
yet according to R2, thus $e_v$ is not recolored.

We distinguish three subcases based on the height of $w_i$.

$\ast$ \emph{Subcase 2.1}: $h(w_i)\equiv 0 \ (mod\ 3)$

If $h(v)=h(w_i)-1$, by Fact $1$ we know that $v$ is already flagged.
No matter whether $e_v$ is recolored or not, we color $w_iv$ with $5$.
Then $c(P_1^{w_i})\cup c(P_2^{w_i})\cup c(P_3^{w_i})=\{2\}\cup
\{1,4\}\cup \{3,5,6\}$. Now $w_i$ becomes safe. We flag $w_i$.

If $h(v)=h(w_i)$, then $v$ may be flagged and $e_v$ may be recolored.
If $e_v$ is not recolored ($c(e_v)=2$), then color $w_iv$ with $6$
and recolor $e_{w_i}$ with $5$. The parent of $w_i$ is still safe.
Now $c(P_1^{p(w_i)})\cup c(P_2^{p(w_i)})\cup
c(P_3^{p(w_i)})=\{1\}\cup \{3,6\}\cup \{4,5\}$. Moreover,
$c(P_1^{w_i})\cup c(P_2^{w_i})\cup c(P_3^{w_i})=\{5\}\cup
\{1,4\}\cup \{2,6\}$ and $c(P_1^{v})\cup c(P_2^{v})\cup
c(P_3^{v})=\{2\}\cup \{1,4\}\cup \{5,6\}$, i.e. both $w_i$ and $v$
are safe. We flag $w_i$ and $v$ (if $v$ is not flagged). If $e_v$ is
recolored ($c(e_v)=5$), it implies $v\in A$ has been dealt with and
got flagged. Then color $w_iv$ with $6$. Now $c(P_1^{w_i})\cup
c(P_2^{w_i})\cup c(P_3^{w_i})=\{2\}\cup \{1,4\}\cup \{5,6\}$, i.e.
$w_i$ becomes safe. We flag $w_i$.

If $h(v)=h(w_i)+1$, by Fact $2$ we know that $e_v$ is not recolored
($c(e_v)=3$). Then color $w_iv$ with $6$. We have $c(P_1^{w_i})\cup
c(P_2^{w_i})\cup c(P_3^{w_i})=\{2\}\cup \{1,4\}\cup \{3,6\}$ and
$c(P_1^{v})\cup c(P_2^{v})\cup c(P_3^{v})=\{3\}\cup \{2,5\}\cup
\{1,4,6\}$, i.e. both $w_i$ and $v$ are safe. We flag $w_i$ and $v$
(if $v$ is not flagged).

$\ast$ \emph{Subcase 2.2}: $h(w_i)\equiv 1 \ (mod\ 3)$

If $h(v)=h(w_i)-1$, by Fact $1$ we know that $v$ is already flagged.
No matter whether $e_v$ is recolored or not, we color $w_iv$ with $6$.
Then $c(P_1^{w_i})\cup c(P_2^{w_i})\cup c(P_3^{w_i})=\{3\}\cup
\{2,5\}\cup \{1,4,6\}$. Now $w_i$ becomes safe. We flag $w_i$.

If $h(v)=h(w_i)$, then $v$ may be flagged and $e_v$ may be
recolored. If $e_v$ is not recolored ($c(e_v)=3$), then color $w_iv$
with $4$ and recolor $e_{w_i}$ with $6$. The parent of $w_i$ is
still safe. Now $c(P_1^{p(w_i)})\cup c(P_2^{p(w_i)})\cup
c(P_3^{p(w_i)})=\{2\}\cup \{1,4\}\cup \{5,6\}$. Moreover,
$c(P_1^{w_i})\cup c(P_2^{w_i})\cup c(P_3^{w_i})=\{6\}\cup
\{2,5\}\cup \{3,4\}$ and $c(P_1^{v})\cup c(P_2^{v})\cup
c(P_3^{v})=\{3\}\cup \{2,5\}\cup \{4,6\}$, i.e. both $w_i$ and $v$
are safe. We flag $w_i$ and $v$ (if $v$ is not flagged). If $e_v$ is
recolored ($c(e_v)=6$), it implies $v$ is flagged. Then color $w_iv$
with $4$. Now $c(P_1^{w_i})\cup c(P_2^{w_i})\cup
c(P_3^{w_i})=\{3\}\cup \{2,5\}\cup \{4,6\}$, i.e. $w_i$ becomes
safe. We flag $w_i$.

If $h(v)=h(w_i)+1$, by Fact $2$ we know that $e_v$ is not recolored
($c(e_v)=1$). Then color $w_iv$ with $4$. We have $c(P_1^{w_i})\cup
c(P_2^{w_i})\cup c(P_3^{w_i})=\{3\}\cup \{2,5\}\cup \{1,4\}$ and
$c(P_1^{v})\cup c(P_2^{v})\cup c(P_3^{v})=\{1\}\cup \{3,6\}\cup
\{2,4,5\}$, i.e. both $w_i$ and $v$ are safe. We flag $w_i$ and $v$
(if $v$ is not flagged).

$\ast$ \emph{Subcase 2.3}: $h(w_i)\equiv 2 \ (mod\ 3)$

If $h(v)=h(w_i)-1$, by Fact $1$ we know that $v$ is already flagged.
No matter whether $e_v$ is recolored or not, we color $w_iv$ with $4$.
Then $c(P_1^{w_i})\cup c(P_2^{w_i})\cup c(P_3^{w_i})=\{1\}\cup
\{3,6\}\cup \{2,4,5\}$. Now $w_i$ becomes safe. We flag $w_i$.

If $h(v)=h(w_i)$, then $v$ may be flagged and $e_v$ may be
recolored. If $e_v$ is not recolored ($c(e_v)=1$), then color $w_iv$
with $5$ and recolor $e_{w_i}$ with $4$. The parent of $w_i$ is
still safe. Now $c(P_1^{p(w_i)})\cup c(P_2^{p(w_i)})\cup
c(P_3^{p(w_i)})=\{3\}\cup \{2,5\}\cup \{1,4,6\}$. Moreover,
$c(P_1^{w_i})\cup c(P_2^{w_i})\cup c(P_3^{w_i})=\{4\}\cup
\{3,6\}\cup \{1,5\}$ and $c(P_1^{v})\cup c(P_2^{v})\cup
c(P_3^{v})=\{1\}\cup \{3,6\}\cup \{4,5\}$, i.e. both $w_i$ and $v$
are safe. We flag $w_i$ and $v$ (if $v$ is not flagged). If $e_v$ is
recolored ($c(e_v)=4$), it implies $v$ is flagged. Then color $w_iv$
with $5$. Now $c(P_1^{w_i})\cup c(P_2^{w_i})\cup
c(P_3^{w_i})=\{1\}\cup \{3,6\}\cup \{4,5\}$, i.e. $w_i$ becomes
safe. We flag $w_i$.

If $h(v)=h(w_i)+1$, by Fact $2$ we know that $e_v$ is not recolored
($c(e_v)=2$). Then color $w_iv$ with $5$. Now $c(P_1^{w_i})\cup
c(P_2^{w_i})\cup c(P_3^{w_i})=\{1\}\cup \{3,6\}\cup \{2,5\}$ and
$c(P_1^{v})\cup c(P_2^{v})\cup c(P_3^{v})=\{2\}\cup \{1,4\}\cup
\{3,5,6\}$, i.e. both $w_i$ and $v$ are safe. We flag $w_i$ and $v$
(if $v$ is not flagged).

\emph{\textbf{Case 3}}: $\pi(w_i)=v_j(1\leq j\leq k-1)$ and there
exists at least one uncolored edge connecting $w_i$ to some subtree
of type $I$. Then we choose one such edge $w_iv$ such that the
height of $v$ is as small as possible. Since $T$ is a $BFS$-tree,
then $h(v)=h(w_i)-1$, $h(w_i)$, or $h(w_i)+1$. We have the following
two facts, which are similar to $Fact$ $1$ and $2$:

\noindent\textbf{Fact $2'$.} If $h(v)=h(w_i)-1$, then $v$ is already
flagged.

If $v\notin A$, then $v$ gets flagged at the very beginning; if $v\in A$,
let $\pi(v)=v_{j'}$, since $1\leq j\leq j'\leq k-1$ and
$h(v)<h(w_i)$, we have already dealt with $v$ according to R1 and
R2, thus $v$ is flagged (note that $e_v$ may be recolored).

\noindent\textbf{Fact $3'$.} If $h(v)=h(w_i)+1$, then $e_v$ is not recolored.

If $v\notin A$, then $e_v$ never gets recolored; if $v\in A$, let
$\pi(v)=v_{j'}$, since $1\leq j'\leq j\leq k-1$ and $h(v)>h(w_i)$,
we have not dealt with $v$ yet according to R1 and R2, thus $e_v$ is
not recolored.

We distinguish three subcases based on the height of $w_i$.

$\ast$ \emph{Subcase 3.1}: $h(w_i)\equiv 0 \ (mod\ 3)$

If $h(v)=h(w_i)-1$, by Fact $2'$ we know that $v$ is already
flagged. No matter whether $e_v$ is recolored or not, we color $w_iv$
with $4$. Then $c(P_1^{w_i})\cup c(P_2^{w_i})\cup
c(P_3^{w_i})=\{2\}\cup \{3,6\}\cup \{1,4,5\}$. Now
$w_i$ becomes safe. We flag $w_i$.

If $h(v)=h(w_i)$, then $v$ may be flagged and $e_v$ may be
recolored. If $e_v$ is not recolored ($c(e_v)=2$), then color $w_iv$
with $5$ and recolor $e_{w_i}$ with $4$. The parent of $w_i$ is
still safe. Now $c(P_1^{p(w_i)})\cup c(P_2^{p(w_i)})\cup
c(P_3^{p(w_i)})=\{3\}\cup \{1,5\}\cup \{4,6\}$. Moreover,
$c(P_1^{w_i})\cup c(P_2^{w_i})\cup c(P_3^{w_i})=\{4\}\cup
\{3,6\}\cup \{2,5\}$ and $c(P_1^{v})\cup c(P_2^{v})\cup
c(P_3^{v})=\{2\}\cup \{3,6\}\cup \{4,5\}$, i.e. both $w_i$ and $v$
are safe. We flag $w_i$ and $v$ (if $v$ is not flagged). If $e_v$ is
recolored ($c(e_v)=4$), it implies $v$ is flagged. Then color $w_iv$
with $5$. Now $c(P_1^{w_i})\cup c(P_2^{w_i})\cup
c(P_3^{w_i})=\{2\}\cup \{3,6\}\cup \{4,5\}$, i.e. $w_i$ becomes
safe. We flag $w_i$.

If $h(v)=h(w_i)+1$, by Fact $3'$ we know that $e_v$ is not recolored
($c(e_v)=1$). Then color $w_iv$ with $5$. Now $c(P_1^{w_i})\cup
c(P_2^{w_i})\cup c(P_3^{w_i})=\{2\}\cup \{3,6\}\cup \{1,5\}$ and
$c(P_1^{v})\cup c(P_2^{v})\cup c(P_3^{v})=\{1\}\cup \{2,4\}\cup
\{3,5,6\}$, i.e. both $w_i$ and $v$ are safe. We flag $w_i$ and $v$
(if $v$ is not flagged).

$\ast$ \emph{Subcase 3.2}: $h(w_i)\equiv 1 \ (mod\ 3)$

If $h(v)=h(w_i)-1$, by Fact $2'$ we know that $v$ is already
flagged. No matter whether $e_v$ is recolored or not, we color $w_iv$
with $5$. Then $c(P_1^{w_i})\cup c(P_2^{w_i})\cup
c(P_3^{w_i})=\{1\}\cup \{2,4\}\cup \{3,5,6\}$. Now
$w_i$ becomes safe. We flag $w_i$.

If $h(v)=h(w_i)$, then $v$ may be flagged and $e_v$ may be
recolored. If $e_v$ is not recolored ($c(e_v)=1$), then $e_v$ never
gets recolored in second step. For $w_i$ not in the first level,
color $w_iv$ with $6$ and recolor $e_{w_i}$ with $5$. The parent of
$w_i$ is still safe. Now $c(P_1^{p(w_i)})\cup c(P_2^{p(w_i)})\cup
c(P_3^{p(w_i)})=\{2\}\cup \{3,6\}\cup \{4,5\}$. Moreover,
$c(P_1^{w_i})\cup c(P_2^{w_i})\cup c(P_3^{w_i})=\{5\}\cup
\{2,4\}\cup \{1,6\}$ and $c(P_1^{v})\cup c(P_2^{v})\cup
c(P_3^{v})=\{1\}\cup \{2,4\}\cup \{5,6\}$, i.e. both $w_i$ and $v$
are safe. We flag $w_i$ and $v$ (if $v$ is not flagged). For $w_i$
in the first level, the parent of $w_i$, namely $v_0$, is already
safe. Its three internally disjoint super-rainbow paths to $D$ are
$P_1^{v_0}=v_0,t(v_0)$, $P_2^{v_0}=v_0,v,t(v)$,
$P_3^{v_0}=v_0,v_k,t(v_k)$, and $c(P_1^{v_0})\cup c(P_2^{v_0})\cup
c(P_3^{v_0})=\{2\}\cup \{1,4\}\cup \{3,5\}$ or $\{2\}\cup
\{1,4\}\cup \{5,6\}$. Since the paths
$P_1^{v_0},P_2^{v_0},P_3^{v_0}$ do not use $e_{w_i}$, we can recolor
$e_{w_i}$ with an arbitrary color from $\{1,2,3,4,5,6\}$. In line
with the previous case, we also color $w_iv$ with $6$ and recolor
$e_{w_i}$ with $5$. Then again both $w_i$ and $v$ are safe. We flag
$w_i$ and $v$ (if $v$ is not flagged). If $e_v$ is recolored
($c(e_v)=5$), it implies $v$ is flagged. Then color $w_iv$ with $6$.
Now $c(P_1^{w_i})\cup c(P_2^{w_i})\cup c(P_3^{w_i})=\{1\}\cup
\{2,4\}\cup \{5,6\}$, i.e. $w_i$ becomes safe. We flag $w_i$.

If $h(v)=h(w_i)+1$, by Fact $3'$ we know that $e_v$ is not recolored
($c(e_v)=3$). Then color $w_iv$ with $6$. Now $c(P_1^{w_i})\cup
c(P_2^{w_i})\cup c(P_3^{w_i})=\{1\}\cup \{2,4\}\cup \{3,6\}$ and
$c(P_1^{v})\cup c(P_2^{v})\cup c(P_3^{v})=\{3\}\cup \{1,5\}\cup
\{2,4,6\}$, i.e. both $w_i$ and $v$ are safe. We flag $w_i$ and $v$
(if $v$ is not flagged).

$\ast$ \emph{Subcase 3.3}: $h(w_i)\equiv 2 \ (mod\ 3)$

If $h(v)=h(w_i)-1$, by Fact $2'$ we know that $v$ is already
flagged. No matter whether $e_v$ is recolored or not, we color $w_iv$
with $6$. Then $c(P_1^{w_i})\cup c(P_2^{w_i})\cup
c(P_3^{w_i})=\{3\}\cup \{1,5\}\cup \{2,4,6\}$. Now
$w_i$ becomes safe. We flag $w_i$.

If $h(v)=h(w_i)$, then $v$ may be flagged and $e_v$ may be
recolored. If $e_v$ is not recolored ($c(e_v)=3$), then color $w_iv$
with $4$ and recolor $e_{w_i}$ with $6$. The parent of $w_i$ is
still safe. Now $c(P_1^{p(w_i)})\cup c(P_2^{p(w_i)})\cup
c(P_3^{p(w_i)})=\{1\}\cup \{2,4\}\cup \{5,6\}$. Moreover,
$c(P_1^{w_i})\cup c(P_2^{w_i})\cup c(P_3^{w_i})=\{6\}\cup
\{1,5\}\cup \{3,4\}$ and $c(P_1^{v})\cup c(P_2^{v})\cup
c(P_3^{v})=\{3\}\cup \{1,5\}\cup \{4,6\}$, i.e. both $w_i$ and $v$
are safe. We flag $w_i$ and $v$ (if $v$ is not flagged). If $e_v$ is
recolored ($c(e_v)=6$), it implies $v$ is flagged. Then color $w_iv$
with $4$. Now $c(P_1^{w_i})\cup c(P_2^{w_i})\cup
c(P_3^{w_i})=\{3\}\cup \{1,5\}\cup \{4,6\}$, i.e. $w_i$ becomes
safe. We flag $w_i$.

If $h(v)=h(w_i)+1$, by Fact $3'$ we know that $e_v$ is not recolored
($c(e_v)=2$). Then color $w_iv$ with $4$. Now $c(P_1^{w_i})\cup
c(P_2^{w_i})\cup c(P_3^{w_i})=\{3\}\cup \{1,5\}\cup \{2,4\}$ and
$c(P_1^{v})\cup c(P_2^{v})\cup c(P_3^{v})=\{2\}\cup \{3,6\}\cup
\{1,4,5\}$, i.e. both $w_i$ and $v$ are safe. We flag $w_i$ and $v$
(if $v$ is not flagged).

\emph{\textbf{Case 4}}: $\pi(w_i)=v_j(1\leq j\leq k-1)$ and all the
uncolored edges connect $w_i$ to the subtree of type
$\uppercase\expandafter{\romannumeral2}$. Then we choose one such
edge $w_iv$ such that the height of $v$ is as small as possible.
Since $T$ is a $BFS$-tree and the subtree of $w_i$ is to the left of
the subtree of $v$, then $h(v)=h(w_i)-1$ or $h(w_i)$. We have the
following fact:

\noindent\textbf{Fact 4.} $v$ is already flagged.

If $v\notin A$, then $v$ gets flagged at the very beginning; if $v\in A$,
since $\pi(v)=v_k$ and $\pi(w_i)=v_j\ (1\leq j\leq k-1)$, we have
already dealt with $v$ according to R1, thus $v$ is flagged (note
that $e_v$ may be recolored).

We distinguish three subcases based on the height of $w_i$.

$\ast$ \emph{Subcase 4.1}: $h(w_i)\equiv 0 \ (mod\ 3)$

If $h(v)=h(w_i)-1$, by Fact $4$ we know that $v$ is already flagged.
If $e_v$ is not recolored ($c(e_v)=1$), then color $w_iv$ with $5$.
We have $c(P_1^{w_i})\cup c(P_2^{w_i})\cup c(P_3^{w_i})=\{2\}\cup
\{3,6\}\cup \{1,5\}$. If $e_v$ is recolored ($c(e_v)=4$), then color
$w_iv$ with $5$. We have $c(P_1^{w_i})\cup c(P_2^{w_i})\cup
c(P_3^{w_i})=\{2\}\cup \{3,6\}\cup \{4,5\}$. Now $w_i$ becomes safe.
We flag $w_i$.

If $h(v)=h(w_i)$, by Fact $4$ we know that $v$ is already flagged.
If $e_v$ is not recolored ($c(e_v)=2$), then color $w_iv$ with $5$.
We have $c(P_1^{w_i})\cup c(P_2^{w_i})\cup c(P_3^{w_i})=\{2\}\cup
\{3,6\}\cup \{1,4,5\}$. If $e_v$ is recolored ($c(e_v)=5$), then
color $w_iv$ with $4$. We have $c(P_1^{w_i})\cup c(P_2^{w_i})\cup
c(P_3^{w_i})=\{2\}\cup \{3,6\}\cup \{4,5\}$. Now $w_i$ becomes safe.
We flag $w_i$.

$\ast$ \emph{Subcase 4.2}: $h(w_i)\equiv 1 \ (mod\ 3)$

If $h(v)=h(w_i)-1$, by Fact $4$ we know that $v$ is already flagged.
If $e_v$ is not recolored ($c(e_v)=2$), then color $w_iv$ with $5$.
We have $c(P_1^{w_i})\cup c(P_2^{w_i})\cup c(P_3^{w_i})=\{1\}\cup
\{3,4,6\}\cup \{2,5\}$. If $e_v$ is recolored ($c(e_v)=5$), then
color $w_iv$ with $6$. We have $c(P_1^{w_i})\cup c(P_2^{w_i})\cup
c(P_3^{w_i})=\{1\}\cup \{2,4\}\cup \{5,6\}$. Now $w_i$ becomes safe.
We flag $w_i$.

If $h(v)=h(w_i)$, by Fact $4$ we know that $v$ is already flagged.
If $e_v$ is not recolored ($c(e_v)=3$), then color $w_iv$ with $6$.
We have $c(P_1^{w_i})\cup c(P_2^{w_i})\cup c(P_3^{w_i})=\{1\}\cup
\{2,4\}\cup \{3,6\}$. If $e_v$ is recolored ($c(e_v)=6$), then color
$w_iv$ with $3$. We have $c(P_1^{w_i})\cup c(P_2^{w_i})\cup
c(P_3^{w_i})=\{1\}\cup \{2,4\}\cup \{3,6\}$. Now $w_i$ becomes safe.
We flag $w_i$.

$\ast$ \emph{Subcase 4.3}: $h(w_i)\equiv 2 \ (mod\ 3)$

If $h(v)=h(w_i)-1$, by Fact $4$ we know that $v$ is already flagged.
If $e_v$ is not recolored ($c(e_v)=3$), then color $w_iv$ with $4$
and recolor $e_{w_i}$ with $6$. The parent of $w_i$ is still safe.
Now $c(P_1^{p(w_i)})\cup c(P_2^{p(w_i)})\cup
c(P_3^{p(w_i)})=\{1\}\cup \{2,4\}\cup \{5,6\}$. Moreover,
$c(P_1^{w_i})\cup c(P_2^{w_i})\cup c(P_3^{w_i})=\{6\}\cup
\{1,5\}\cup \{3,4\}$ , i.e. $w_i$ is safe. We flag $w_i$. If $e_v$
is recolored ($c(e_v)=6$), then color $w_iv$ with $4$. Now
$c(P_1^{w_i})\cup c(P_2^{w_i})\cup c(P_3^{w_i})=\{3\}\cup
\{1,5\}\cup \{4,6\}$, i.e. $w_i$ becomes safe. We flag $w_i$.

If $h(v)=h(w_i)$, by Fact $4$ we know that $v$ is already flagged.
If $e_v$ is not recolored ($c(e_v)=1$), then color $w_iv$ with $3$
and recolor $e_{w_i}$ with $6$. The parent of $w_i$ is still safe.
Now $c(P_1^{p(w_i)})\cup c(P_2^{p(w_i)})\cup
c(P_3^{p(w_i)})=\{1\}\cup \{2,4\}\cup \{5,6\}$. Moreover,
$c(P_1^{w_i})\cup c(P_2^{w_i})\cup c(P_3^{w_i})=\{6\}\cup
\{2,4,5\}\cup \{1,3\}$, i.e. $w_i$ is safe. We flag $w_i$. If $e_v$
is recolored ($c(e_v)=4$), then color $w_iv$ with $6$. Now
$c(P_1^{w_i})\cup c(P_2^{w_i})\cup c(P_3^{w_i})=\{3\}\cup
\{1,5\}\cup \{4,6\}$, i.e. $w_i$ becomes safe. We flag $w_i$.

Then we go to $w_{i+1}$ and repeat the process, until all the
vertices in $A$ are visited. We do the same operation to all
$C_{i}$$'s$. If there still exist uncolored edges in
$E[D,\overline{D}]\cup E(G[\overline{D}])$, color them with $1$. Now
we have a coloring of all the edges in $E[D,\overline{D}]\cup
E(G[\overline{D}])$ using six different colors from
$\{1,2,3,4,5,6\}$ such that all the vertices in $\overline{D}$ are
safe.

\subsection{Color the edges in $E(G[D])$}
Let $d:=rx_3(G[D])$. Then we can color the edges in $G[D]$ with $d$
fresh colors from $\{7,8,\ldots, d+6\}$ such that for each triple of
vertices in $D$, there exists a rainbow tree in $G[D]$ connecting
them. Hereto we obtain an edge-coloring $c: E(G)\rightarrow
\{1,2,\ldots,d+6\}$.

\subsection{Prove $c$ is a 3-rainbow coloring}
Next we will prove that this edge-coloring of $G$ is a 3-rainbow
coloring, which yields that $rx_3(G)\leq rx_3(G[D])+6$.

\begin{claim}
Under this coloring, for any three vertices $u,v,w$ in
$\overline{D}$, there exists a rainbow $u-D$ path $P^u$, a rainbow
$v-D$ path $P^v$ and a rainbow $w-D$ path $P^w$ such that $P^u \cup
P^v \cup P^w$ is also rainbow.
\end{claim}
Before giving the proof of Claim 1, let us show how it implies our
result. Let $S=\{u,v,w\}\subseteq V(G)$. If $|S\cap D|=3$, i.e.
$(u,v,w)\in D\times D \times D$, then there is already a rainbow
$S$-tree in $G[D]$. If $|S\cap D|=2$, say $(u,v,w)\in D\times
D\times \overline{D}$, then let $w'$ be the foot of $w$. The rainbow
tree in $G[D]$ connecting $u,v,w'$ together with the edge $ww'$
forms a rainbow $S$-tree. If $|S\cap D|=1$, say $(u,v,w)\in D \times
\overline{D} \times \overline{D}$, then by $Claim\ 1$, there exists
a rainbow $v-D$ path $P^v$ and a rainbow $w-D$ path $P^w$ such that
$P^v \cup P^w$ is also rainbow. Assume the endvertex of $P^v$, $P^w$
in $D$ is $v'$, $ w'$ respectively. Then the rainbow tree in $G[D]$
connecting $u,v',w'$ together with the paths $P^v$ and $P^w$ forms a
connected rainbow subgraph of $G$, denoted by $H$. Obviously, a
spanning tree of $H$ is a rainbow $S$-tree. If $|S\cap D|=0$, i.e.
$(u,v,w)\in \overline{D} \times \overline{D}\times \overline{D}$,
then by $Claim\ 1$, there exists a rainbow $u-D$ path $P^u$, a
rainbow $v-D$ path $P^v$ and a rainbow $w-D$ path $P^w$ such that
$P^u\cup P^v \cup P^w$ is also rainbow. Assume the endvertex of
$P^u$, $P^v$, $P^w$ in $D$ is $u'$, $v'$, $ w'$ respectively. Then
the rainbow tree in $G[D]$ connecting $u',v',w'$ together with the
paths $P^u$, $P^v$ and $P^w$ forms a connected rainbow subgraph of
$G$, denoted by $H'$. Obviously, a spanning tree of $H'$ is a
rainbow $S$-tree. So we come to the conclusion that the
edge-coloring $c$ is a 3-rainbow coloring.

\noindent\emph{Proof of Claim 1}: For any three vertices $u,v,w$ in
$\overline{D}$, $u,v,w$ are safe under this coloring. That is, there
exist three internally-disjoint super-rainbow $u-D$ paths $P^u_1$,
$P^u_2$, $P^u_3$, three internally-disjoint super-rainbow $v-D$
paths $P^v_1$, $P^v_2$, $P^v_3$ and three internally-disjoint
super-rainbow $w-D$ paths $P^w_1$, $P^w_2$, $P^w_3$. If we can pick
out $P_i^u$, $P_j^v$ and $P_k^w$ $(1\leq i,j,k\leq 3)$ from these
paths satisfying $P_i^u\cup P_j^v \cup P_k^w$ is also rainbow, we
are done. But unfortunately in some cases, we can not do that. For
example, if $c(P_1^u)\cup c(P_2^u) \cup c(P_3^u)= \{1\} \cup \{2,4\}
\cup \{5,6\}$, $c(P_1^v)\cup c(P_2^v) \cup c(P_3^v)= \{1\} \cup
\{2,5\} \cup \{4,6\}$, $c(P_1^w)\cup c(P_2^w) \cup c(P_3^w)= \{1\}
\cup \{2,6\} \cup \{4,5\}$, then one can check that $P_i^u\cup P_j^v
\cup P_k^w$ is not rainbow for each $1\leq i,j,k\leq 3$. Here we
will show a sufficient and necessary condition for the situation in
which we can pick out suitable $P_i^u$, $P_j^v$ and $P_k^w$. (Note
that $P_1^u$, $P_1^v$, $P_1^w$ contains exactly one edge.)

There exist $i,j,k\in\{1,2,3\}$ satisfying $P_i^u\cup P_j^v \cup
P_k^w$ is rainbow if and only if

(C1) $c(P_1^u)$, $c(P_1^v)$, $c(P_1^w)$ are not the same or

(C2) there exist two distinct vertices $x,y\in \{u,v,w\}$ and two
integers $s,t\in\{2,3\}$ ($s$ may equal to $t$) such that
$c(P_s^x)\cap c(P_t^y)=\emptyset$.

If (C1) is true, without loss of generality, we assume $c(P_1^u)=1$
and $c(P_1^v)=2$ . If $c(P_1^w)\in\{3,4,5,6\}$, then $P_1^u\cup
P_1^v \cup P_1^w$ is rainbow; If $c(P_1^w)\in\{1,2\}$, without loss
of generality let $c(P_1^w)=1$. Since $P_1^w\cup P_2^w \cup P_3^w$
is rainbow, both $P_2^w$ and $P_3^w$ contain no edges colored by 1,
and at least one of $P_2^w$ and $P_3^w$ contains no edges colored by
2, say $P_2^w$. Then $P_1^u\cup P_1^v \cup P_2^w$ is rainbow. If
(C2) is true, without loss of generality, we assume that
$c(P_2^u)\cap c(P_2^v)=\emptyset$. If $c(P_1^u)$, $c(P_1^v)$,
$c(P_1^w)$ are not the same, then the assertion holds by (C1);
otherwise, without loss of generality let
$c(P_1^u)=c(P_1^v)=c(P_1^w)=1$. Then $P_2^u$ and $P_2^v$ contain no
edges colored by 1. Bearing in mind that $c(P_2^u)\cap
c(P_2^v)=\emptyset$, we get that $P_1^w\cup P_2^u \cup P_2^v$ is
rainbow. For the other direction, assume that (C1) is not true, we
will show (C2) holds by contradiction. Suppose that
$c(P_1^u)=c(P_1^v)=c(P_1^w)$, and for any two distinct vertices
$x,y\in \{u,v,w\}$ and any two integers $s,t\in\{2,3\}$,
$c(P_s^x)\cap c(P_t^y)\neq \emptyset$. Since $P_i^u\cup P_j^v \cup
P_k^w$ is rainbow, we know that at most one of $i,j,k$ is equal to
1, say $j,k\in\{2,3\}$. Then by hypothesis, $c(P_j^v)\cap
c(P_k^w)\neq \emptyset$, a contradiction to the fact that $P_i^u\cup
P_j^v \cup P_k^w$ is rainbow.

From the above assertion, we can see that the colors of the three
internally disjoint super-rainbow paths connecting a vertex in
$\overline{D}$ to $D$ plays a crucial role. Here we list out all the
possible color sets of these paths under this coloring. For the sake
of brevity, we write $\{1,24,35\}$ instead of $c(P_1^u)\cup
c(P_2^u)\cup c(P_3^u)=\{1\}\cup \{2,4\} \cup \{3,5\}$ and
$c(P_1^u)\cup c(P_2^u)\cup c(P_3^u)=\{1\}\cup \{3,5\} \cup \{2,4\}$.
\\

\emph{Class 0}: $\{1,2,3\}$, ~~$\{1,2,34\}$, ~~$\{1,2,36\}$, ~~$\{2,3,14\}$, ~~$\{2,3,15\}$,

 ~~~~~~~~~~~~$\{1,3,24\}$, ~~$\{1,3,25\}$

\emph{Class 1}: $\{1,24,35\}$, ~~$\{1,36,24\}$, ~~$\{1,36,25\}$, ~~$\{1,24,56\}$, ~~$\{1,36,45\}$,

~~~~~~~~~~~~$\{1,36,245\}$, ~~$\{1,24,356\}$, ~~$\{1,346,25\}$.

\emph{Class 2}: $\{2,36,14\}$, ~~$\{2,14,35\}$, ~~$\{2,14,56\}$, ~~$\{2,36,15\}$, ~~$\{2,36,45\}$,

~~~~~~~~~~~~$\{2,46,35\}$, ~~$\{2,36,145\}$, ~~$\{2,14,356\}$, ~~$\{2,346,15\}$.

\emph{Class 3}: $\{3,15,26\}$, ~~$\{3,25,16\}$, ~~$\{3,15,46\}$, ~~$\{3,25,46\}$, ~~$\{3,15,24\}$,

~~~~~~~~~~~~$\{3,25,14\}$, ~~$\{3,25,146\}$, ~~$\{3,15,246\}$.

\emph{Class 4}: $\{4,36,15\}$, ~~$\{4,36,25\}$, ~~$\{4,36,125\}$.

\emph{Class 5}: $\{5,14,26\}$, ~~$\{5,24,16\}$.

\emph{Class 6}: $\{6,25,34\}$, ~~$\{6,15,34\}$, ~~$\{6,15,24\}$, ~~$\{6,245,13\}$.

For every triple $\{u,v,w\}$ of vertices in $\overline{D}$, if
$c(P_1^u)$, $c(P_1^v)$ and $c(P_1^w)$ are not the same, we are done.
Now suppose $c(P_1^u)=c(P_1^v)=c(P_1^w)$. If there exists one vertex
satisfying at least two of its three paths are of length 1, without
loss of generality, we assume that $c(P_1^u)=c(P_1^v)=c(P_1^w)=1$,
$P_2^u$ is of length 1, and $c(P_2^u)=2$. Since $P_1^v\cup P_2^v\cup
P_3^v$ is rainbow, we can find out one path, say $P_2^v$, which
contains no edges colored by 1 or 2. Then $P_2^u \cup P_2^v \cup
P_1^w$ is rainbow. Again we are done. Thus to prove $Claim \ 1$, it
suffices to check whether (C2) holds for every three color sets in
Class $i$ ($1\leq i\leq 6$). Since the number of color sets in one
class is no more than 9, the checking work can be done in a short
time and the answer in turn is affirmative. We complete the proof of
$Claim \ 1$.

To end the section, we illustrate the tightness of the bound
$rx_3(G)\leq rx_3(G[D])+6$ with the graph in $Figure \ 5$. It is
easy to see that $D=\{v_0\}$ is a connected three-way dominating
set. By Theorem \ref{thm3}, $rx_3(G)\leq rx_3(G[D])+6=6$. On the
other hand, we have already proved that $rx_3(G)=6$. So the bound is
tight.

\section{Concluding remarks}
To sum up, as for the 3-rainbow index of a graph, we can consider
the following three strengthened connected dominating sets:

Let $G$ be a connected graph and $D$ be a connected dominating set
of $G$.

(a) if every vertex in $\overline{D}$ is adjacent to at least three
distinct vertices of $D$, then $rx_3(G)\leq rx_3(G[D])+3$ (Theorem
\ref{thm4});

(b) if every vertex in $\overline{D}$ is of degree at least three
and adjacent to at least two distinct vertices of $D$, then
$rx_3(G)\leq rx_3(G[D])+4$ (Theorem \ref{thm2});

(c) if every vertex in $\overline{D}$ is of degree at least three,
then $rx_3(G)\leq rx_3(G[D])+6$ (Theorem \ref{thm3}).

From (a) to (c), we loosen the restrictions on the connected
dominating sets, while the additive constant increases. We cannot
tell which bound is the best. For example, for a French Windmill in
Figure 5, (c) is better than (a) and (b), whereas for a threshold
graph with $\delta\geq3$, (a) and (b) which imply $rx_3(G)\leq5$ are
better than (c) which implies $rx_3(G)\leq6$. Given a connected
graph $G$, we can calculate three upper bounds for the 3-rainbow
index of $G$ using (a), (b), (c) respectively (some of them may be
the same), and then choose the smallest one of them.

\end{document}